\documentclass[reqno, 12pt]{amsart}

\numberwithin{equation}{section}

\usepackage{amsfonts,amssymb,amsmath,amsthm}
\usepackage{enumerate}
\usepackage{bbm}
\usepackage{times}
\usepackage{amsfonts}
\usepackage{amssymb}
\usepackage{bbm}
\usepackage{times}
\usepackage{amsfonts}
\usepackage{amssymb}
\usepackage{bbm}
\usepackage{times}

\usepackage{enumerate}
\usepackage{amsmath}
\usepackage{amsthm}
\usepackage{color}

\usepackage [latin2]{inputenc}
\newtheorem{thm}{Theorem}[section]
\newtheorem{lem}[thm]{Lemma}

\newtheorem{prop}[thm]{Proposition}

\newcounter{other}            
\newtheorem{otherth}[other]{Theorem}              

\newcommand{\aut}{{\rm Aut}(\D)}
\def\B{\mathcal{B}}
\def\D{\mathbb{D}}
\def\C{\mathbb{C}}

\newcommand{\cn}{\C^n}

\def\ind{\int_\D}

\def\Q{\mathcal Q}

\def\ind{\int_\D}

\newcommand{\mo}{M\"obius\ }

\begin{document}

\title[Duality for $\alpha$-M\"obius invariant Besov  spaces]
{Duality for $\alpha$-M\"obius invariant Besov  spaces}

\author{Guanlong Bao}
\address{Department of Mathematics\\
         Shantou University\\
         Shantou 515063, Guangdong, China}
\email{glbao@stu.edu.cn}

\author{Zengjian Lou}
\address{Department of Mathematics\\
         Shantou University\\
         Shantou 515063, Guangdong, China}
\email{zjlou@stu.edu.cn}

\author{Xiaojing Zhou}
\address{Department of Mathematics\\
         Shantou University\\
         Shantou 515063, Guangdong, China}
\email{18xjzhou@stu.edu.cn}

\thanks{G. Bao   was supported  by National Natural Science Foundation  of China (No. 12271328) and Guangdong Basic and Applied Basic Research Foundation (No. 2022A1515012117). Z. Lou and X. Zhou was supported  by National Natural Science Foundation  of China (No. 12071272) }



\begin{abstract}
For $1\leq p\leq \infty$ and $\alpha>0$,  Besov spaces  $B^p_\alpha$ play a key role in the theory of   $\alpha$-M\"obius invariant function spaces. In some sense,  $B^1_\alpha$ is the minimal $\alpha$-M\"obius invariant function space, $B^2_\alpha$ is the unique $\alpha$-M\"obius invariant Hilbert space, and $B^\infty_\alpha$ is the maximal $\alpha$-M\"obius invariant function space.  In this paper, under the  $\alpha$-M\"obius invariant pairing and by the space $B^\infty_\alpha$,
we identify the  predual and dual spaces of $B^1_\alpha$. In particular, the corresponding identifications are isometric isomorphisms. The duality theorem via the  $\alpha$-M\"obius invariant pairing for $B^p_\alpha$ with $p>1$ is also given.

\end{abstract}

\maketitle

\section{Introduction}

 Let $\D$ be the open unit disk in the complex plane. Denote by
$\aut$ the M\"obius  group of  one-to-one analytic  functions that maps $\D$ onto itself.  For $a\in \D$, let
$$
\sigma_a(z)=\frac{a-z}{1-\overline{a}z}, \qquad z\in \D,
$$
be   a M\"obius map  of $\D$ interchanging the points $0$ and $a$.  It is well known  that
$$
\text{Aut}(\D)=\{e^{i\theta}\sigma_a:\  \ a\in \D \ \text{and} \ \theta \ \ \text{is a real number}\}.
$$
A space  $X$ contained in  $H(\D)$,  the set  of functions analytic in $\D$,  is said to be  M\"obius invariant if it is equipped with a semi-norm $\rho$ such that   $f\circ\varphi\in X$ and
$\rho(f\circ\varphi)\lesssim\rho(f)$ for all $f\in X$ and all $\varphi\in\aut$.
The study of the theory of M\"obius invariant function spaces is a  classical topic in complex analysis (cf. \cite{AF, AFP, AXZ,  RT, WZ,  X1, X2}). Roughly speaking,  it is known from \cite{AF, AFP, RT} that  the
Bloch space $\B$, the Dirichlet space $\mathcal D$, and the Besov space $B^1$ give us  the  maximal  M\"obius invariant function space, the unique \mo invariant Hilbert space, and the minimal M\"obius invariant function space, respectively.

For $\alpha>0$,  let $\B_\alpha$ be the Bloch type space consisting of those functions $f\in H(\D)$ satisfying
$$
|||f|||_{\B_\alpha}=\sup_{z\in \D}(1-|z|^2)^\alpha |f'(z)|<\infty.
$$
$|||\cdot|||_{\B_\alpha}$ is a semi-norm on $\B_\alpha$.  Denote by $\B_{\alpha, 0}$  the closure of polynomials in $\B_\alpha$.
If $\alpha=1$, then $\B_\alpha$ is the Bloch space $\B$. Because of the maximal property of $\B$ among  M\"obius invariant function spaces,  K. Zhu \cite{Zhu1} posed  the question of whether $\B_\alpha$ is maximal among some family of analytic function spaces.  To answer K. Zhu's question and understand the theory of a general family of analytic function spaces $F(p, q, s)$,  in 1996 R. Zhao  \cite{Zhao} introduced a notion of weighted composition  of $f$ in $H(\D)$ with $\varphi$ in $\aut$. For $\alpha>0$, $f\in H(\D)$, and $\varphi\in\aut$, let
$$
f\circ_\alpha\varphi(z)=\int_0^zf'(\varphi(w))(\varphi'(w))^\alpha\,dw+f(\varphi(0))(\varphi'(0))^{\alpha-1}
$$
for $z\in\D$. Clearly, if $\alpha=1$, then  $f\circ_\alpha\varphi=f\circ\varphi$, the usual composition of $f$ and $\varphi$. Suppose $X$ is a subspace of  $H(\D)$ equipped with a semi-norm $\rho$.  Following Section 4 in \cite{Zhao}, we say that  $X$ is
$\alpha$-\mo invariant if  $f\circ_\alpha \varphi\in X$ and
$\rho(f\circ_\alpha\varphi)\le C\rho(f)$ for all $f\in X$ and all $\varphi\in\aut$, where $C$ is a positive constant independent of $f$ and $\varphi$. In fact,
if   the condition above holds for  a semi-norm $\rho$, then there is  another  equivalent semi-norm $\rho'$ on $X$
satisfying that  $\rho'(f\circ_\alpha\varphi)=\rho'(f)$ for all $f\in X$ and $\varphi\in\aut$. Also, when $\alpha=1$, an $\alpha$-\mo invariant function space is a \mo invariant function space. Indeed, among $\alpha$-\mo invariant function spaces,  R. Zhao  \cite{Zhao} proved that   $\B_\alpha$  is maximal in some sense. In \cite[p. 54]{Zhao}, R. Zhao posed the following two questions:  which space is the minimal $\alpha$-\mo invariant function space? Is there one $\alpha$-\mo invariant Hilbert space only? The two questions are answered in \cite{BZ} recently. We refer to \cite{AM, KU1, KU2} for some recent  study  related to M\"obius invariant function spaces.

 For $\beta >-1$, we write $dA_\beta (z)=(\beta +1)(1-|z|^2)^\beta  dA(z)$, where $dA(z)=1/\pi dxdy$, $z=x+iy$,  is  the normalized Lebesgue measure on $\D$.
Suppose  $1\leq p < \infty$, $\alpha>0$, and $n$ is a positive integer with $p(\alpha-1+n)-1>0$. Recall that the Besov space  $B^p_\alpha$ is the set of those functions $f$ in $H(\D)$ satisfying that the function
$f^{(n)}(z)(1-|z|^2)^{\alpha-1+n}$ is in $L^p(\D, d\lambda)$, where $d\lambda(z)=(1-|z|^2)^{-2}dA(z)$ is the \mo invariant  Lebesgue measure on $\D$. Also, for $\alpha>0$ and a positive integer $n$, $B^\infty_\alpha$ consists of functions $f$ in $H(\D)$ such that $f^{(n)}(z)(1-|z|^2)^{\alpha-1+n}$ is bounded on $\D$. It is known from \cite{ZZ} that $B^p_\alpha$ is independent of the choice of the integer $n$. Clearly, $B^\infty_\alpha$ is the Bloch type space $\B_\alpha$.  By \cite{Zhao},  $|||f\circ_\alpha \varphi|||_{\B_\alpha}=|||f|||_{\B_\alpha}$ for every $f\in \B_\alpha$ and  $\varphi \in \aut$.  From Theorem 6.2 in \cite{BZ}, for any $p\geq 1$ and $\alpha>0$, the space $B^p_\alpha$ is also  $\alpha$-\mo invariant.

For  $\alpha>0$,  denote by  $M_\alpha$ the space of those functions $f\in H(\D)$ that can be
represented as
\begin{equation}\label{eq4}
f(z)=c_0+\sum_{n=1}^\infty c_n\int_0^z\varphi'_n(w)^\alpha\,dw,\qquad z\in\D,
\end{equation}
where $c_0\in\C$, $\{c_n\}\in\ell^1$, and every  $\varphi_n\in\aut$.  A semi-norm on $M_\alpha$ is given by
$$|||f|||_{M_\alpha}=\inf\left\{\sum_{n=1}^\infty |c_n|: \ \ (\ref{eq4})\  \text{holds}\right\}.$$
It is known from \cite{BZ} that $|||f\circ_\alpha \varphi|||_{M_\alpha}=|||f|||_{M_\alpha}$ for every $f\in M_\alpha$ and  $\varphi \in \aut$, and $M_\alpha$ is minimal among all non-trivial  $\alpha$-\mo invariant function spaces.
From Theorem 3.2 in  \cite{BZ},
\begin{equation}\label{Ma equa}
|||f|||_{M_\alpha}\thickapprox |f'(0)|+ \ind|f''(z)|(1-|z|^2)^{\alpha-1}\,dA(z).
\end{equation}
for all $f\in M_\alpha$. A norm of $f$ in $M_\alpha$ is $\|f\|_{M_\alpha}=|f(0)|+|||f|||_{M_\alpha}$. Because of  (\ref{Ma equa}), $M_\alpha$ is the Besov space $B^1_\alpha$.

Given  $\alpha>0$ and a function $f(z)=\sum_{n=0}^\infty a_n z^n$  analytic in $\D$,  we say that $f\in H_\alpha$ if
$$
|||f|||_{H_\alpha}= \sum_{n=1}^\infty\frac{n!\,\Gamma(2\alpha)}{\Gamma(n+2\alpha-1)} n |a_n|^2<\infty.
$$
From \cite{BZ},  $|||f\circ_\alpha \varphi|||_{H_\alpha}=|||f|||_{H_\alpha}$ for each  $f\in H_\alpha$ and  $\varphi \in \aut$, and $H_\alpha$ is the unique
non-trivial $\alpha$-\mo invariant Hilbert space. It is also known from \cite{BZ} that  $H_\alpha$ is the Besov space $B^2_\alpha$.

Fix $\alpha>0$,  related to the space $H_\alpha$, a paring on $\D$ is given by
\begin{equation}\label{paring}
\langle f,g\rangle_\alpha= \lim_{r\rightarrow 1^-}\sum_{n=1}^\infty\frac{n!\,\Gamma(2\alpha)}{\Gamma(n+2\alpha-1)}
\,na_n\overline b_n r^{2n},
\end{equation}
if the limit exists, where $f(z)=\sum_{n=0}^\infty a_nz^n$ and $g(z)=\sum_{n=0}^\infty b_nz^n$ are functions analytic in $\D$.  In fact, it is known from \cite{BZ} that  $$
\langle f\circ_\alpha\varphi,g\circ_\alpha\varphi\rangle_\alpha=\langle f,g\rangle_\alpha
$$
for all $\varphi\in\aut$. Thus we say that $\langle \cdot, \cdot\rangle_\alpha$ is an $\alpha$-M\"obius invariant pairing.

In this paper, under the $\alpha$-M\"obius invariant pairing, we give  the dual relation between the minimal and the maximal $\alpha$-M\"obius invariant function spaces.
The corresponding identifications are isometric isomorphisms. In particular, we prove that  there exists $f(z)=\sum_{n=0}^\infty a_n z^n$ in $B_{\alpha, 0}$ and $g(z)=\sum_{n=0}^\infty b_n z^n$ in $M_{\alpha}$ such that
$$
\sum_{n=1}^\infty\frac{n!\,\Gamma(2\alpha)}{\Gamma(n+2\alpha-1)}
\,n \overline a_n b_n
$$
is divergent, which means that the definition of  $\langle \cdot, \cdot\rangle_\alpha$ is reasonable via a limit.  We also investigate  the duality theorem for $B^p_\alpha$ when  $p>1$.

Throughout  this paper,   we write $a \lesssim b$ if   there exists a positive constant $C$ such that $a \leq Cb$. If  $a\lesssim b\lesssim a$, then we write  $a \thickapprox b$.

\section{$\alpha$-M\"obius invariant pairing and a general  duality result}

In this section, we give  some equalities of the $\alpha$-M\"obius invariant pairing and  a general  duality theorem for $\alpha$-M\"obius invariant function spaces. These results will be useful in  next sections.

\begin{lem} \label{a mo pair}
Suppose both $f(z)=\sum_{n=0}^\infty a_nz^n$ and $g(z)=\sum_{n=0}^\infty b_nz^n$ belong to $H(\D)$, and the limit $\lim_{r\rightarrow 1^-}\sum_{n=1}^\infty\frac{n!\,\Gamma(2\alpha)}{\Gamma(n+2\alpha-1)}\,na_n\overline b_n r^{2n}$ exists.  Then the following statements hold.
\begin{itemize}
  \item [(a)] For  $\alpha>1/2$,
  \begin{equation}\label{paring integral}
\langle f,g\rangle_\alpha= (2\alpha-1) \lim_{r\rightarrow 1^-}\ind(1-|z|^2)^\alpha f'(rz)\overline{(1-|z|^2)^\alpha g'(rz)}\,d\lambda(z).
\end{equation}
\item [(b)] For  $\alpha>0$,
\begin{eqnarray}
\nonumber \langle f,g\rangle_\alpha&= &\lim_{r\rightarrow 1^-}\Big[\ind f'(rz)\overline{g'(rz)}dA_{2\alpha-1}(z)\\
&~&+\frac{\Gamma(2\alpha)}{\Gamma(2\alpha+2)} r^2\ind f''(rz)\overline{g''(rz)}dA_{2\alpha}(z)\Big].  \label{another integral paring1}
\end{eqnarray}
\end{itemize}
\end{lem}
\begin{proof}
For $r\in (0, 1)$,  replacing $f$ and $g$  in \cite[Lemma 4.2 and Lemma 4.3]{BZ} by $f_r$ and $g_r$ respectively, we get the desired results.  Here $f_r(z)=f(rz)$ and $g_r(z)=g(rz)$.
\end{proof}

Next, by the  $\alpha$-M\"obius invariant pairing, we give a description of the dual space of an  $\alpha$-\mo invariant function space.

\begin{thm}\label{dualgeneral}
Suppose  $\alpha>0$ and $X$ is a function space equipped with  an $\alpha$-\mo invariant semi-norm  $|||\cdot|||_{X}$. Assume that polynomials are dense in $X$.
Then the dual space $(X/\C)^*$ of $X/\C$ can be identified with the space  of  functions
$f\in H(\D)$ satisfying
$$
|||f|||_{(X/\C)^*}=\sup\{|\langle g,f\rangle_\alpha|: g\in X, |||g|||_{X} \leq 1\}<\infty.
$$
 Moreover, $(X/\C)^*$ is an $\alpha$-\mo invariant function space.
\end{thm}
\begin{proof}
Let  $f\in H(\D)$ with $|||f|||_{(X/\C)^*}<\infty$. Clearly, $\langle \cdot,f\rangle_\alpha$ gives a bounded linear functional on $X/\C$.

Now let $F$ be  a bounded linear functional on $X$ with $F(1)=0$.
Write
\begin{equation}\label{an F}
a_n=\overline{F(z^n)} \frac{\Gamma(n+2\alpha-1)}{n! \,\Gamma(2\alpha) n}, \ \ n=1, 2, 3, \cdots ,
\end{equation}
 and set $f(z)=\sum_{n=1}^{+\infty} a_n z^n$ for $z\in \D$.  Then
$$
|a_n|\leq \frac{\Gamma(n+2\alpha-1)}{n! \,\Gamma(2\alpha) n} \|z^n\|_{X}\|F\|,
$$
where $\|\cdot\|_X$ is a  norm of $X$, and $\|F\|$ is the norm of the bounded linear functional  $F$.
Note that $M_\alpha$ is the minimal non-trivial $\alpha$-\mo invariant function space. By Theorem 3.3 in \cite{BZ},
$$
\|z^n\|_{X}\lesssim \|z^n\|_{M_\alpha}
$$
for all positive integers $n$. Clearly, $\|z\|_{M_\alpha}$ is a positive constant.
For $n\geq 2$, a direct computation yields
$$
\|z^n\|_{M_\alpha}\approx n(n-1)\ind|z|^{n-2}(1-|z|^2)^{\alpha-1}\,dA(z)\approx n^{2-\alpha},
$$
where the comparable positive constants depend only on $\alpha$. Consequently,
 $$
 |a_n|\lesssim \frac{\Gamma(n+2\alpha-1)}{n! \,\Gamma(2\alpha) n}n^{2-\alpha} \|F\|
 $$
 for all positive integers $n$. Note that the function  $\sum_{n=1}^{+\infty} \frac{\Gamma(n+2\alpha-1)}{n! \,\Gamma(2\alpha) n}n^{2-\alpha} z^n$ belongs to $H(\D)$. Thus $f\in H(\D)$.
For any polynomial $g(z)=\sum_{n=0}^N b_n z^n$, where $N$ is a positive integer, it follows from (\ref{an F}) that
\begin{eqnarray*}
F(g)&=& b_0F(1)+\sum_{n=1}^N b_n F(z^n)\\
&=& \sum_{n=1}^N b_n \overline{a_n} \frac{n! \,\Gamma(2\alpha) n}{\Gamma(n+2\alpha-1)}\\
&=& \langle g, f\rangle_\alpha.
\end{eqnarray*}
Since  polynomials are dense in $X$, we see that $F(h)= \langle h,f\rangle_\alpha$ for all $h\in X$. Thus $(X/\C)^*$ can be identified with the space of functions $f\in H(\D)$ with $|||f|||_{(X/\C)^*}<\infty$.

The $\alpha$-M\"obius invariance of $\langle \cdot,\cdot\rangle_\alpha$ and $|||\cdot|||_{X}$ yields that $|||\cdot|||_{(X/\C)^*}$ is an $\alpha$-M\"obius invariant semi-norm of $(X/\C)^*$. The proof is complete.
\end{proof}

 Theorem \ref{dualgeneral}   generalizes  the classical duality result in \cite{AFP}   from    \mo invariant function spaces to $\alpha$-\mo invariant function spaces for each   $\alpha>0$.

\section{Duality theorems associated with  the minimal and maximal $\alpha$-M\"obius invariant function spaces }

In this section, for  every $\alpha>0$,  we show that the dual  space of $M_\alpha/\C$ can be identified with $\B_\alpha/\C$,  and the predual space  of $M_\alpha/\C$ can be  identified with $\B_{\alpha, 0}/\C$.  The corresponding identifications  are isometric isomorphisms. For the clear  presentation, we will state these  results by the notion of  bounded linear functional.

We first consider the existence of the limit in (\ref{paring}) as follows.
\begin{lem}\label{31}
Suppose $\alpha>0$.
For   $f(z)=\sum_{n=0}^\infty a_nz^n$ in $\B_\alpha$  and  $g(z)=\sum_{n=0}^\infty b_nz^n$ in   $M_\alpha$,    let
$$
h(r)=\sum_{n=1}^\infty \frac{n!\,\Gamma(2\alpha)}{\Gamma(n+2\alpha-1)}
\,n a_n \overline{b_n}r^{2n}.
$$
Then the limit $\lim_{r\to 1^-}h(r)$ exists.
\end{lem}
\begin{proof}
  For $r\in (1/2, 1)$, it follows from  (\ref{another integral paring1}) that
\begin{align*}
h(r)=\ind f'(rz)\overline{g'(rz)}dA_{2\alpha-1}(z)+\frac{\Gamma(2\alpha)}{\Gamma(2\alpha+2)} r^2\ind f''(rz)\overline{g''(rz)}dA_{2\alpha}(z).
\end{align*}
Combining this with the characterization of weighted Bergman spaces and Bloch type spaces  via higher order derivatives (cf. \cite{Zhu, Zhu1}), we see that
\begin{align}\label{exists lim}
|h(r)|\lesssim &|||f|||_{\B_\alpha}\ind |g'(rz)| dA_{\alpha-1}(z) \nonumber \\
&+\frac{\Gamma(2\alpha)}{\Gamma(2\alpha+2)} |||f|||_{\B_\alpha} \ind |g''(rz)| dA_{\alpha-1}(z) \nonumber \\
\lesssim &|||f|||_{\B_\alpha} \ind |g''(rz)| dA_{\alpha}(z)  +\frac{\Gamma(2\alpha)}{\Gamma(2\alpha+2)} |||f|||_{\B_\alpha} |||g_r|||_{M_\alpha} \nonumber \\
\lesssim &|||f|||_{\B_\alpha} |||g_r|||_{M_\alpha}\leq  C|||f|||_{\B_\alpha} |||g|||_{M_\alpha},
\end{align}
where $C$ is a positive constant depending only on $\alpha$.

Let $r_1$, $r_2$ be any real numbers in  $(1/2, 1)$. Bear in mind (\ref{exists lim}) and
$$
g_{r_1}(z)-g_{r_2}(z)=\sum_{n=0}^\infty b_n (r_1^n-r_2^n)z^n.
$$
We deduce
\begin{align*}
|h(r_1)-h(r_2)|=& \left|\sum_{n=1}^\infty \frac{n!\,\Gamma(2\alpha)}{\Gamma(n+2\alpha-1)}
\,n a_n \overline{b_n}(r_1^{2n}-r_2^{2n})\right| \\
=& \lim_{s \to 1^-} \left|\sum_{n=1}^\infty \frac{n!\,\Gamma(2\alpha)}{\Gamma(n+2\alpha-1)}
\,n a_n \overline{b_n}(r_1^{2n}-r_2^{2n})s^{2n}\right|\\
\leq & C|||f|||_{\B_\alpha} |||g_{r_1}-g_{r_2}|||_{M_\alpha}.
\end{align*}
Note that $|||g_{r_1}-g_{r_2}|||_{M_\alpha}$ tends to $0$ uniformly as $|r_1-r_2|\to 0$. Then $|h(r_1)-h(r_2)| \to 0$  uniformly as $|r_1-r_2|\to 0$. Thus the limit $\lim_{r\to 1^-}h(r)$ exists.
\end{proof}

The following result means that  the dual  space of $M_\alpha/\C$ can be identified with $\B_\alpha/\C$ by isometric isomorphisms.
\begin{thm}\label{1main}
For    $\alpha>0$,   the following statements are true.
\begin{itemize}
  \item [(a)]  For any  $f\in \B_\alpha$, $ \langle \cdot,f\rangle_\alpha$ defines a bounded linear functional on $M_\alpha$.  Moreover,
 \begin{equation}\label{1equial}
  |||f|||_{\B_\alpha}=\sup\{|\langle g,f\rangle_\alpha|: g\in M_\alpha, |||g|||_{M_\alpha} \leq 1\}.
\end{equation}
  \item [(b)] If $F$ is a bounded linear functional on $M_\alpha$ with $F(1)=0$, then there exits a function $f\in \B_\alpha$ such that
  $F(g)= \langle g,f\rangle_\alpha$ for all $g\in M_\alpha$.  Also,
   \begin{equation}\label{bubu1}
  |||f|||_{\B_\alpha}=\sup\{|F(g)|: g\in M_\alpha, |||g|||_{M_\alpha} \leq 1\}.
\end{equation}
   \end{itemize}
\end{thm}
\begin{proof} (a)\ \  Let   $f\in \B_\alpha$. Because of Lemma \ref{exists lim},  $ \langle g,f\rangle_\alpha$ is well defined for any $g\in M_\alpha$. It is clear that $ \langle \cdot,f\rangle_\alpha$ defines a linear functional on $M_\alpha$.

 For $\alpha>0$  and $\zeta\in \D$,
 we claim   that
\begin{equation}\label{claim formula}
\left\langle I_{\alpha, \zeta}, f \right\rangle_\alpha=(|\zeta|^2-1)^\alpha \overline{f'(\zeta)},
\end{equation}
where
$$
I_{\alpha, \zeta}(z)=: \int_0^z\sigma'_\zeta(w)^\alpha\,dw, \ \ z\in \D.
$$
Consider the case of  $\alpha>\frac{1}{2}$ first. From  (\ref{paring integral}) and the well-known reproducing formula (cf. \cite[Proposition 4.23]{Zhu}), we get that
\begin{eqnarray*}
\left\langle f, I_{\alpha, \zeta} \right\rangle_\alpha&=&  (2\alpha-1)(|\zeta|^2-1)^\alpha \lim_{r\rightarrow 1^-}\ind  \frac{f'(rz)(1-|z|^2)^{2\alpha-2}}{(1-r\zeta \overline{z})^{2\alpha}}dA(z)\\
&=& (|\zeta|^2-1)^\alpha \lim_{r\rightarrow 1^-} f'(r^2\zeta)\\
&=& (|\zeta|^2-1)^\alpha  f'(\zeta),
\end{eqnarray*}
which yields formula  (\ref{claim formula}).  Next we consider the case of  all $\alpha>0$. By (\ref{another integral paring1}),
\begin{align*}
\left\langle f, I_{\alpha, \zeta} \right\rangle_\alpha= &(|\zeta|^2-1)^\alpha  \lim_{r\rightarrow 1^-}\Big[\ind \frac{f'(rz)}{(1-r\zeta \overline{z})^{2\alpha}} dA_{2\alpha-1}(z)\\
&+\frac{\Gamma(2\alpha)}{\Gamma(2\alpha+2)} 2\alpha \zeta r^2\ind \frac{f''(rz)}{(1-r\zeta \overline{z})^{2\alpha+1}}dA_{2\alpha}(z)\Big].
\end{align*}
 Note that $f\in \B_\alpha$ and $\zeta \in \D$.   Lebesgue's  Dominated
Convergence Theorem yields
\begin{eqnarray*}
\left\langle f, I_{\alpha, \zeta} \right\rangle_\alpha
&=& (|\zeta|^2-1)^\alpha \Big[\ind \frac{f'(z)}{(1-\zeta \overline{z})^{2\alpha}} dA_{2\alpha-1}(z)\\
 &~&+\frac{\Gamma(2\alpha)}{\Gamma(2\alpha+2)} 2\alpha \zeta \ind \frac{f''(z)}{(1-\zeta \overline{z})^{2\alpha+1}}dA_{2\alpha}(z) \Big]\\
& =:& G(\alpha).
\end{eqnarray*}
Denote by $\Omega=\{z\in \C: \Re (z) >0\}$ the right-half plane. It is well known  that the Gamma function is analytic on $\Omega$. Hence $G$ extends to an analytic function on $\Omega$. For fixed $\zeta \in \D$, we write   $G_1(\beta)=(|\zeta|^2-1)^\beta  f'(\zeta)$.  Clearly, $G_1$ is also an analytic function on  $\Omega$. We have shown that $G(\beta)=G_1(\beta)$ when $\beta$ is any real number in  $(\frac{1}{2}, +\infty)$. From the uniqueness property of analytic functions, we get $G_1(\beta)=G(\beta)$  for all $\beta \in \Omega$. Hence our claim  holds when $\alpha>0$ and $\zeta \in \D$.

Suppose $g\in M_\alpha$. Then  there exists $c_0\in\C$, $\{c_n\}\in\ell^1$ and $\{\zeta_n\}\subseteq \D$ such that
$$
g(z)=c_0+\sum_{n=1}^\infty c_n I_{\alpha, \zeta_n}(z),\qquad z\in\D.
$$
Consequently,
\begin{eqnarray*}
|\langle g,f\rangle_\alpha|&\leq&\sum_{n=1}^\infty |c_n||\langle f, I_{\alpha, \zeta_n} \rangle_\alpha|\\
&\leq &   \sum_{n=1}^\infty |c_n| \sup_{\zeta \in \D}|\langle f, I_{\alpha, \zeta} \rangle_\alpha|.
\end{eqnarray*}
Taking the  infimum over all representations for $g$, we get
\begin{equation}\label{mid}
|\langle g,f\rangle_\alpha|\leq |||g|||_{M_\alpha} \sup_{\zeta \in \D}|\langle f, I_{\alpha, \zeta} \rangle_\alpha|.
\end{equation}
It follows from  (\ref{claim formula}) and (\ref{mid}) that
\begin{equation}\label{mid1}
|\langle g,f\rangle_\alpha|\leq |||g|||_{M_\alpha} ||||f|||_{\B_\alpha}.
\end{equation}
Thus $ \langle \cdot,f\rangle_\alpha$ defines a bounded linear functional on $M_\alpha$.

Clearly, (\ref{mid1}) implies
$$
\sup\{|\langle h, f\rangle_\alpha|: h\in M_\alpha, |||h|||_{M_\alpha} \leq 1\}\leq  ||||f|||_{\B_\alpha}.
$$
Note that $|||I_{\alpha, \zeta}|||_{M_\alpha}\leq 1$ for any $\zeta\in \D$. Then
$$
 ||||f|||_{\B_\alpha}=\sup_{\zeta \in \D}|\langle f, I_{\alpha, \zeta} \rangle_\alpha|\leq \sup\{|\langle h, f\rangle_\alpha|: h\in M_\alpha, |||h|||_{M_\alpha} \leq 1\}.
$$
Hence
$$
|||f|||_{\B_\alpha}=\sup\{|\langle h, f\rangle_\alpha|: h\in M_\alpha, |||h|||_{M_\alpha} \leq 1\}.
$$

(b)\ \   Suppose $F$ is  a bounded linear functional on $M_\alpha$ with $F(1)=0$.
Set
$$
a_n=\overline{F(z^n)} \frac{\Gamma(n+2\alpha-1)}{n! \,\Gamma(2\alpha) n}, \ \ n=1, 2, 3, \cdots ,
$$
 and let $f(z)=\sum_{n=1}^{+\infty} a_n z^n$ for $z\in \D$.
Note that polynomials are dense in $M_\alpha$. By the proof of Proposition \ref{dualgeneral}, we see that $f\in H(\D)$ and  $F(h)= \langle h,f\rangle_\alpha$ for all $h\in M_\alpha$.

 For  any $\zeta\in \D$ and $z\in \D$,
\begin{eqnarray*}
I_{\alpha, \zeta}(z)&=& (|\zeta|^2-1)^\alpha \int_0^z \frac{1}{(1-\overline{\zeta} w)^{2\alpha}}dw\\
&=& (|\zeta|^2-1)^\alpha \sum_{n=1}^{+\infty} \frac{\Gamma(n+2\alpha-1)}{n! \,\Gamma(2\alpha) } \overline{\zeta}^{n-1}z^n.
\end{eqnarray*}
Consequently,
\begin{eqnarray*}
\nonumber F(I_{\alpha, \zeta}(z))&=&(|\zeta|^2-1)^\alpha \sum_{n=1}^{+\infty} \frac{\Gamma(n+2\alpha-1)}{n! \,\Gamma(2\alpha) } \overline{\zeta}^{n-1} F(z^n)\\
\nonumber  &=& (|\zeta|^2-1)^\alpha \sum_{n=1}^{+\infty}  n \overline{a_n}\overline{\zeta}^{n-1}\\
&=& (|\zeta|^2-1)^\alpha \overline{f'(\zeta)}.
\end{eqnarray*}
Thus,
$$
|||f|||_{\B_\alpha}=\sup_{\zeta\in \D}|F(I_{\alpha, \zeta})|\leq \|F\|\sup_{\zeta\in \D}|||I_{\alpha, \zeta}|||_{M_\alpha}\leq ||F||.
$$
Hence $f\in \B_\alpha$. Since $F(g)= \langle g,f\rangle_\alpha$ for all $g\in M_\alpha$,  it follows from (\ref{1equial}) that  (\ref{bubu1}) also  holds.  We finish the proof.
\end{proof}

The follows theorem gives that the predual space  of $M_\alpha/\C$ can be  identified with $\B_{\alpha, 0}/\C$.  The corresponding identification is also an  isometric isomorphism.

\begin{thm}\label{2main}
For    $\alpha>0$,  the following statements hold.
\begin{itemize}
  \item [(a)] For any  $f\in M_\alpha$, $ \langle \cdot,f\rangle_\alpha$ defines a bounded linear functional on $\B_{\alpha, 0}$. Moreover,
 \begin{equation}\label{2equial}
  |||f|||_{M_\alpha}=\sup\{|\langle h, f\rangle_\alpha|: h\in \B_{\alpha, 0}, |||h|||_{\B_\alpha} \leq 1\}.
\end{equation}
\item [(b)] If $F$ is a bounded linear functional on $\B_{\alpha, 0}$ with $F(1)=0$, then there exits a function $f\in M_\alpha$ such that
  $F(h)= \langle h,f\rangle_\alpha$ for all $h\in \B_{\alpha, 0}$.  Also,
  \begin{equation}\label{bubu2}
  |||f|||_{M_\alpha}=\sup\{|F(h)|: h\in \B_{\alpha, 0}, |||h|||_{\B_\alpha} \leq 1\}.
\end{equation}
\end{itemize}
\end{thm}
\begin{proof} (a)\ \
   Bear in mind that $\B_{\alpha, 0}$ is a subset of $\B_\alpha$.  For  $f\in M_\alpha$,  from the proof of (a) of Theorem \ref{1main}, one gets  immediately that
$ \langle \cdot,f\rangle_\alpha$ defines a bounded linear functional on $\B_{\alpha, 0}$.  By a well-known corollary of the Hahn-Banach theorem (cf. \cite[Corollary 4.8.6]{Fr}), we know
\begin{align*}
|||f|||_{M_\alpha}&= \sup \Big\{ \frac{|\phi(f)|}{\|\phi\|_{(M_\alpha/\C)^*}}: \\
&  \phi\  \text{is a  bounded linear functional on}\  M_\alpha/\C,\  \phi(1)=0,  \ \text{and} \ \phi\not=0 \Big\}.
\end{align*}
Combining this with Theorem \ref{1main}, we get
\begin{equation}\label{9mid1}
|||f|||_{M_\alpha}=\sup_{h\not =0} \frac{|\langle h,f\rangle_\alpha|}{|||h|||_{\B_\alpha}}
\end{equation}
Thus,
\begin{equation}\label{9mid2}
\sup\{|\langle h, f\rangle_\alpha|: h\in \B_{\alpha, 0}, |||h|||_{\B_\alpha} \leq 1\} \leq |||f|||_{M_\alpha}.
\end{equation}
Note that  $ \langle \cdot,f\rangle_\alpha$ defines a bounded linear functional on $\B_{\alpha, 0}$. Also, given $r\in (0, 1)$, every $h_r$ belongs to $\B_{\alpha, 0}$  if $h\in \B_\alpha$. Consequently, for any $0<r<1$ and  $h\in \B$,
\begin{align*}
|\langle h, f_r\rangle_\alpha|=&|\langle h_r, f\rangle_\alpha|\\
  \leq & |||h_r|||_{\B_\alpha} \sup\{|\langle g, f\rangle_\alpha|: g\in \B_{\alpha, 0}, |||g|||_{\B_\alpha} \leq 1\}\\
  \leq & |||h|||_{\B_\alpha} \sup\{|\langle g, f\rangle_\alpha|: g\in \B_{\alpha, 0}, |||g|||_{\B_\alpha} \leq 1\}.
\end{align*}
 This together with (\ref{9mid1}) yields that
$$
|||f_r|||_{M_\alpha}\leq \sup\{|\langle g, f\rangle_\alpha|: g\in \B_{\alpha, 0}, |||g|||_{\B_\alpha} \leq 1\},
$$
which implies
\begin{equation}\label{9mid3}
|||f|||_{M_\alpha}\leq \sup\{|\langle g, f\rangle_\alpha|: g\in \B_{\alpha, 0}, |||g|||_{\B_\alpha} \leq 1\}.
\end{equation}
By (\ref{9mid2}) and  (\ref{9mid3}), we get (\ref{2equial}).

(b)\ \   Write
$$
b_n=\overline{F(z^n)} \frac{\Gamma(n+2\alpha-1)}{n! \,\Gamma(2\alpha) n}, \ \ n=1, 2, 3, \cdots ,
$$
 and set  $f(z)=\sum_{n=1}^{+\infty} b_n z^n$ for $z\in \D$.  Bear in mind  that polynomials are dense in $\B_{\alpha, 0}$. From  the proof of Proposition \ref{dualgeneral},  $f\in H(\D)$ and  $F(h)= \langle h,f\rangle_\alpha$ for all $h\in \B_{\alpha, 0}$.

For $0<r<1$, it is clear that  $f_r \in M_\alpha$.  For  $g\in \B_{\alpha}$, it is clear that  $g_r\in \B_{\alpha, 0}$. We deduce hat
 \begin{align*}
 |\langle g, f_r\rangle_\alpha|= |\langle g_r, f\rangle_\alpha|=|F(g_r)|\leq \|F\| |||g_r|||_{\B_\alpha}\leq \|F\| |||g|||_{\B_\alpha}
 \end{align*}
 for all $g\in \B_\alpha$. This together with (\ref{9mid1}) gives
 $ |||f_r|||_{M_\alpha}\leq \|F\|$ for all $r\in (0, 1)$.  Thus $f\in M_\alpha$.  It is known from  (\ref{2equial}) that  (\ref{bubu2}) also   holds.  The proof is finished.
 \end{proof}

\vspace{0.1truecm}
\noindent {\bf  Remark 1.}\ \  Theorem  \ref{1main} and Theorem \ref{2main}  generalize the corresponding   results in \cite{AFP} from  \mo invariant function spaces to $\alpha$-\mo invariant function spaces. But our proof of Theorem \ref{2main} is  different from  \cite{AFP}, where the result was obtained by showing  that $M_1$ is weak-$^*$ dense and weak-$^*$ closed in the dual space of  $\B_{1, 0}/\C$ (cf. \cite[p. 124]{AFP}).

\vspace{0.1truecm}
\noindent {\bf  Remark 2.}\ \  K.  Zhu \cite{Zhu1} established the duality results  between the Bergman space $A^1$ and Bloch type spaces $\B_\alpha$ under another  paring, where the related identifications are with equivalent norms and  not isometry.

\vspace{0.1truecm}
\noindent {\bf  Remark 3.}\ \ It is clear that  the function $f$ in  (b) of both Theorem \ref{1main} and Theorem \ref{2main} is unique in the sense of modulo constants.

By the proof of Theorem \ref{1main}, we see that
\begin{eqnarray*}
|\langle f,g\rangle_\alpha|&=& |\lim_{r\rightarrow 1^-}\sum_{n=1}^\infty\frac{n!\,\Gamma(2\alpha)}{\Gamma(n+2\alpha-1)}
\,na_n\overline b_n r^{2n}|\\
&\leq& |||f|||_{\B_\alpha}|||g|||_{M_\alpha}
\end{eqnarray*}
for all $f\in \B_\alpha$ and $g\in M_\alpha$.  It is natural to ask whether the limit in the definition of  $\langle \cdot, \cdot\rangle_\alpha$ can be dropped by taking $r=1$ in the sum.  Using Theorem \ref{1main}, we give a negative  answer to this question  in Theorem \ref{3main} below.  For a positive integer $k$,  throughout this paper, denote by  $S_k$  the operator sending every function in $H(\D)$ to its $k$-th Taylor polynomial.

\begin{thm}\label{3main}
For  $\alpha>0$, there exists $f(z)=\sum_{n=1}^\infty a_n z^n$ in $\B_{\alpha, 0}$ and $g(z)=\sum_{n=1}^\infty b_n z^n$ in $M_{\alpha}$ such that
$$
\sum_{n=1}^\infty\frac{n!\,\Gamma(2\alpha)}{\Gamma(n+2\alpha-1)}
\,n \overline a_n b_n
$$
is divergent.
\end{thm}
\begin{proof} We follow  the idea in \cite[p. 18]{ACP}.
Suppose the conclusion is not true. For $f(z)=\sum_{n=1}^\infty a_n z^n$ in $\B_{\alpha, 0}$ and every positive integer $k$, define
$$
F_k(g)=\sum_{n=1}^k \frac{n!\,\Gamma(2\alpha)}{\Gamma(n+2\alpha-1)}
\,n \overline a_n b_n,
$$
where   $g(z)=\sum_{n=1}^\infty b_n z^n$ in $M_{\alpha}$. Then every $F_k$  defines a bounded linear functional on $M_\alpha$. Because of our assumption, the uniform boundedness principle
yields that $\{||F_k|| \}$ is a bounded sequence.  Note that
$F_k(g)=\langle g, S_k f\rangle_\alpha$. By Theorem \ref{1main},
$
|||S_k f|||_{\B_\alpha}\leq ||F_k||.
$
Hence $\{|||S_k f|||_{\B_\alpha}\}$ is also a bounded sequence for every $f\in B_{\alpha, 0}$. By the uniform boundedness principle again, we see that
$\{\|S_k \|\}$ is bounded, where $S_k$ is regarded as operator on $\B_{\alpha, 0}$. Thus there exists a positive constant $C$ independing  on $k$ such that
$$
\|S_k f\|_{\B_\alpha}\leq C \| f\|_{\B_\alpha}
$$
for all $f\in B_{\alpha, 0}$. But there exist functions in $\B_{\alpha, 0}$ whose Taylor polynomials divergent in norm (cf. \cite[p. 1159]{Zhu1}). We get a contradiction. This finishes the proof.
\end{proof}

\section{ Duality for $\alpha$-M\"obius invariant $B^p_\alpha$ spaces  with $p>1$ }

For $p>1$ and $\alpha>0$,  this section is devoted to  show that  under the $\alpha$-M\"obius invariant pairing, the dual space of $B^p_\alpha/ \C$ can be identified with $B^q_\alpha/ \C$, where $q$ satisfies $1/p+1/q=1$ and the corresponding  identification  is with equivalent norms.

For the investigation of the Besov space $B^p_\alpha$ with  $p>1$ and $\alpha>0$, the easier case is when $p\alpha>1$. For such case, $B^p_\alpha$ is the space of those functions $f$ in $H(\D)$ satisfying
$$
\int_\D |f'(z)(1-|z|^2)^\alpha|^p d\lambda(z)<\infty.
$$
To cover the case of  $p\alpha<1$, we can describe $B^p_\alpha$ as the space of those functions $f$ in $H(\D)$ such that
\begin{equation}\label{230}
|||f|||_{B^p_\alpha}=\left(\int_\D |f''(z)(1-|z|^2)^{\alpha+1}|^p d\lambda(z)\right)^{1/p}<\infty.
\end{equation}
A norm of $f$ in $B^p_\alpha$ is
$$
\|f\|_{B^p_\alpha}=|f(0)|+|f'(0)|+|||f|||_{B^p_\alpha}.
$$
 For $q>0$ and $\beta>-1$, recall that the Bergman space $A^q_\beta$ consists of functions $f$ in $H(\D)$ with
$$
\|f\|_{A^q_\beta}^q=\int_\D |f(z)|^q (1-|z|^2)^\beta dA(z)<\infty.
$$
By (\ref{230}), for $p>1$ and $\alpha>0$, $f\in B^p_\alpha$ if and only if $f'' \in A^p_{p(\alpha+1)-2}$.

In this section, our proof of the main result involves the property of coefficient multipliers  for Besov spaces.
Given two  spaces $X$ and $Y$ of analytic functions in $\D$, a complex  sequence  $\{\lambda_n\}_{n=0}^\infty$  is said to be a coefficient multiplier from  $X$ to $Y$ if the function $\sum_{n=0}^\infty \lambda_n a_n z^n$ belongs to $Y$ whenever $\sum_{n=0}^\infty  a_n z^n \in X$. Denote by $(X, Y)$ the set of coefficient multipliers from $X$ to $Y$.  We refer to a recent book \cite{JVA} for the theory of coefficient multiplier between   spaces of analytic  functions in  $\D$.

Let $BV$ be the classical space of complex sequences of bounded variation; that is,
$$
BV=\left\{\{\lambda_n\}_{n=0}^\infty: \ \  |\lambda_0|+\sum_{n=0}^\infty |\lambda_{n+1}-\lambda_n|<\infty\right \}.
$$
 The following result is due to S. Buckley, P. Koskela, and D. Vukoti\'c \cite{BKV}.

\begin{otherth}\label{BV X}
Let $X$ be a Banach space of analytic functions in $\D$ such that for every $f\in X$ the  sequence $\{S_n f\}_{n=1}^\infty$ of Taylor polynomials    converges to $f$ in the norm of $X$. Then $BV \subseteq (X, X)$, and the inclusion is strict if the involution $Tf(z)=f(-z)$ is bounded on $X$.
\end{otherth}

We also need the following well-known  result (cf. Corollary 3.13 in  \cite{Zhu}).

\begin{otherth}\label{opera inter}
Suppose $a$, $b$, $\alpha$ are real parameters, $1\leq p<\infty$, and
$$
Tf(z)=(1-|z|^2)^a \ind \frac{(1-|w|^2)^b}{(1-z\overline{w})^{2+a+b}} f(w)dA(w).
$$
 Then the operator $T$ is bounded on $L^p(\D, dA_\alpha)$ if and only if $-pa<\alpha+1<p(b+1)$.
\end{otherth}

For $g\in B_\alpha^p$ and  $f\in B_\alpha^q$, where $\frac{1}{p}+\frac{1}{q}=1$,  we give a new formula of $\langle g, f \rangle_\alpha$ as follows.

\begin{lem} \label{lem 41}
Suppose $p>1$, $\alpha>0$, and $q$ is the  real number with $\frac{1}{p}+\frac{1}{q}=1$. Let  $f(z)=\sum_{n=0}^\infty a_n z^n$ be in  $B_\alpha^q$ and let $g(z)=\sum_{n=0}^\infty b_n z^n$ be in  $B_\alpha^p$.
Then the   following limit
\begin{equation}\label{b280}
\lim_{r\rightarrow 1^-}\sum_{n=1}^\infty\frac{n!\,\Gamma(2\alpha)}{\Gamma(n+2\alpha-1)}\,nb_n\overline a_n r^{2n}
\end{equation}
exists and
\begin{equation}\label{b281}
\langle g, f \rangle_\alpha=b_1\overline a_1 +\ind m_{g, \alpha}''(z)\overline{f''(z)}(1-|z|^2)^{2\alpha}dA(z),
\end{equation}
where
$$
m_{g, \alpha}(z)=\frac{\Gamma(2\alpha)}{\Gamma(2\alpha+1)}\sum_{n=2}^\infty  \frac{n+2\alpha-1}{n-1}  b_n z^n.
$$
\end{lem}
\begin{proof}
 Note that $g\in B^p_\alpha$ if and only if $g'' \in A^p_{p(\alpha+1)-2}$. It follows from \cite[Corollary 4]{Zhu2} that the sequence of Taylor polynomials  of every function in $B^p_\alpha$  converges in norm. Set $\lambda_0=\lambda_1=0$ and $\lambda_n=\frac{n+2\alpha-1}{n-1}$ when  $n=2, \cdots$. It is easy to check that the sequence $\{\lambda_n\}_{n=0}^\infty$ belongs to $BV$.
By Theorem \ref{BV X}, $\{\lambda_n\}_{n=0}^\infty$  is a coefficient multiplier for $B^p_\alpha$.  The closed graph theorem yields $|||m_{g, \alpha}|||_{B^p_\alpha}\lesssim \|g\|_{B^p_\alpha}$ for all $g\in B^p_\alpha$.

For $f\in B^q_\alpha$, set $F_f(g)=\ind m_{g, \alpha}''(z)\overline{f''(z)}(1-|z|^2)^{2\alpha}dA(z)$, where $g\in B^p_\alpha$. Then the H\"older inequality gives that
\begin{align}\label{280}
|F_f(g)|\leq & \left(\ind |f''(z)|^q (1-|z|^2)^{q(\alpha+1)}d\lambda(z)\right)^{1/q} \nonumber \\
  & \times\left(\ind |m_{g, \alpha}''(z)|^p (1-|z|^2)^{p(\alpha+1)}d\lambda(z)\right)^{1/p} \nonumber \\
\leq &  |||f|||_{B^q_\alpha} |||m_{g, \alpha}|||_{B^p_\alpha}\lesssim |||f|||_{B^q_\alpha}  \|g\|_{B^p_\alpha},
\end{align}
which implies  that $F_f$ is a bounded linear functional on $B^p_\alpha$. Note that $\|g_r-g\|_{B^p_\alpha}\to 0$ as $r\to 1^-$. We get $\lim_{r\rightarrow 1^-} F_f(g_r)=F_f(g)$; that is,
\begin{align}\label{281}
&\lim_{r\rightarrow 1^-}\ind m_{g_r, \alpha}''(z)\overline{f''(z)}(1-|z|^2)^{2\alpha}dA(z) \nonumber \\
 =& \ind m_{g, \alpha}''(z)\overline{f''(z)}(1-|z|^2)^{2\alpha}dA(z).
\end{align}
We also see that
\begin{align}\label{282}
&\lim_{r\rightarrow 1^-}\ind m_{g_r, \alpha}''(z)\overline{f''(z)}(1-|z|^2)^{2\alpha}dA(z) \nonumber \\
  =&\lim_{r\rightarrow 1^-} \sum_{n=2}^\infty  \frac{\Gamma(2\alpha)(n+2\alpha-1)}{\Gamma(2\alpha+1)} n^2(n-1)r^n b_n \overline{a_n} \int_0^1 (1-t)^{2\alpha} t^{n-2}dt \nonumber \\
  =& \lim_{r\rightarrow 1^-}  \sum_{n=2}^\infty \frac{n!\,\Gamma(2\alpha)}{\Gamma(n+2\alpha-1)}\,nb_n\overline a_n r^{n} \nonumber \\
  =& \lim_{r\rightarrow 1^-}  \sum_{n=2}^\infty \frac{n!\,\Gamma(2\alpha)}{\Gamma(n+2\alpha-1)}\,nb_n\overline a_n r^{2n}.
\end{align}
Similarly,
\begin{align}\label{283}
 \ind m_{g, \alpha}''(z)\overline{f''(z)}(1-|z|^2)^{2\alpha}dA(z)=\sum_{n=2}^\infty \frac{n!\,\Gamma(2\alpha)}{\Gamma(n+2\alpha-1)}\,nb_n\overline a_n.
\end{align}
Hence,
\begin{align}\label{284}
\lim_{r\rightarrow 1^-}  \sum_{n=2}^\infty \frac{n!\,\Gamma(2\alpha)}{\Gamma(n+2\alpha-1)}\,nb_n\overline a_n r^{2n}=\sum_{n=2}^\infty \frac{n!\,\Gamma(2\alpha)}{\Gamma(n+2\alpha-1)}\,nb_n\overline a_n.
\end{align}
By (\ref{281}) and (\ref{282}),  the limit in (\ref{b280}) exists. By (\ref{283}) and (\ref{284}), we get (\ref{b281}). The proof is complete.
\end{proof}

\vspace{0.1truecm}
\noindent {\bf  Remark 4.}\ \
Let  $p>1$, $\alpha>0$, and let $q$ satisfy  $\frac{1}{p}+\frac{1}{q}=1$. Suppose $f\in B_\alpha^q$ and  $g\in B_\alpha^p$.  From (\ref{284}), the limit in the definition of  $\langle g, f \rangle_\alpha$ can be dropped via taking $r=1$ in the sum. This is different from Theorem \ref{3main}.    One reason of the  difference is that  there are functions in $\B_{\alpha, 0}$ whose Taylor polynomials divergent in norm, and  the sequence of Taylor polynomials  of any  function in $B^p_\alpha$  converges in norm.

From  Lemma \ref{lem 41} and its proof, we get the following conclusion, which is a new representation of   $\langle \cdot , \cdot \rangle_\alpha$.
\begin{prop}\label{ano reper inner}
Suppose $\alpha>0$,  both $f(z)=\sum_{n=0}^\infty a_nz^n$ and $g(z)=\sum_{n=0}^\infty b_nz^n$ belong to $H(\D)$, and the limit $\lim_{r\rightarrow 1^-}\sum_{n=1}^\infty\frac{n!\,\Gamma(2\alpha)}{\Gamma(n+2\alpha-1)}\,nb_n\overline a_n r^{2n}$ exists. Then
$$
 \langle g, f \rangle_\alpha=b_1\overline{a_1}+\lim_{r\rightarrow 1^-} \ind m_{g_r, \alpha}''(z)\overline{f''(z)}(1-|z|^2)^{2\alpha}dA(z).
$$
\end{prop}

For $p>1$, $\alpha>0$, denote by $\widetilde{B}_{\alpha}^{p}$ the set of those functions $h\in B_{\alpha}^{p}$ such that $h(0)=h'(0)=0$.  Now we state and prove the main result in this section.

\begin{thm}\label{dua Bp}
Suppose $p>1$, $\alpha>0$, and $q$ is the  real number with $\frac{1}{p}+\frac{1}{q}=1$.
Then for any  $f\in B_\alpha^q$, $ \langle \cdot,f\rangle_\alpha$ defines a bounded linear functional on $B_\alpha^p$. Conversely,
    if $F$ is a bounded linear functional on $B_\alpha^p$ with $F(1)=0$, then there exits a function $f\in B_\alpha^q$ such that
  $F(g)= \langle g,f\rangle_\alpha$ for all $g\in B_\alpha^p$. Moreover, there exist positive constants $C_1$ and $C_2$ independent of $f$ such that
  \begin{equation}\label{231}
  C_1 |||f|||_{B_\alpha^q}\leq \sup \{|F(g)|:  \ g\in\widetilde{ B}_\alpha^p, \ |||g|||_{B_\alpha^p}\leq 1\}\leq C_2 |||f|||_{B_\alpha^q}.
  \end{equation}
 \end{thm}
\begin{proof}
Let $f\in B_\alpha^q$ and $g\in B_\alpha^p$. By Lemma \ref{lem 41}, the limit in the definition of $ \langle g,f\rangle_\alpha$ exists. Clearly, $ \langle \cdot,f\rangle_\alpha$ defines a  linear functional on $B_\alpha^p$.
From (\ref{b281}) and (\ref{280}),  we get
\begin{align}\label{2101}
\left|\langle g, f \rangle_\alpha\right|\lesssim |f'(0)||g'(0)|+|||f|||_{B^q_\alpha}  \|g\|_{B^p_\alpha}.
\end{align}
Thus the  linear functional $ \langle \cdot,f\rangle_\alpha$ is also bounded on $B_\alpha^p$.

On the other hand, suppose   $F$ is a bounded linear functional on $B_\alpha^p$ with $F(1)=0$.  Let
$$
a_n=\overline{F(z^n)} \frac{\Gamma(n+2\alpha-1)}{n! \,\Gamma(2\alpha) n}, \ \ n=1, 2, 3, \cdots ,
$$
 and let $f(z)=\sum_{n=1}^{+\infty} a_n z^n$ for $z\in \D$.
It is known  that polynomials are dense in $B_\alpha^p$. From  the proof of Proposition \ref{dualgeneral},  $f\in H(\D)$ and  $F(g)= \langle g,f\rangle_\alpha$ for all $g\in B_\alpha^p$.
Next we show that $f$ belongs to $B_\alpha^q$.

For $g\in \widetilde{B}_{\alpha}^{p}$, set
$$
G(g)(z)=(1-|z|^{2})^{\alpha+1}m_{g, \alpha}''(z).
$$
From the proof of Lemma \ref{lem 41}, $|||m_{g, \alpha}|||_{B^p_\alpha}\lesssim |||g|||_{B^p_\alpha}$ for all $g\in B^p_\alpha$. Thus $G$ is a bounded mapping from $\widetilde{B}_{\alpha}^{p}$ to $L^{p}(\D,d\lambda)$.
Write $G(\widetilde{B}_{\alpha}^{p})$ the image of $G$. Then $G: \widetilde{B}_{\alpha}^{p}\rightarrow G(\widetilde{B}_{\alpha}^{p})$ is  bijective. As usual, denote by  $G^{-1}$  the inverse mapping of $G$.
Then  $F \circ G^{-1}$is a bounded linear functional on $G(\widetilde{B}_{\alpha}^{p})$. By the Hahn-Banach extension theorem,  there is a function $\varphi$ in $L^{q}(\D,d\lambda)$  such that
\begin{eqnarray*}
F \circ G^{-1}(k)=\ind k(z)\overline{\varphi(z)}d\lambda(z)
\end{eqnarray*}
for all $k$ in $G(\widetilde{B}_{\alpha}^{p})$, and
\begin{equation}\label{bub2101}
\|F \circ G^{-1}\|=\|\varphi\|_{L^{q}(\D,d\lambda)}.
\end{equation}
 Thus, for $g\in \widetilde{B}_{\alpha}^{p}$, we get
\begin{eqnarray}\label{2102}
F(g)&=&F\circ G^{-1}(G(g))\nonumber \\
&=&\ind G(g)(z)\overline{\varphi(z)}d\lambda(z) \nonumber \\
&=& \ind (1-|z|^{2})^{\alpha+1}m_{g, \alpha}''(z) \overline{\varphi(z)}d\lambda(z).
\end{eqnarray}
Note that
\begin{eqnarray*}
&&\ind |\varphi(z)|(1-|z|^{2})^{\alpha-1}dA(z)\\
&\leq &\ind |\varphi(z)|(1-|z|^{2})^{-1}dA(z)\\
&\leq&\left[\ind(1-|z|^{2})^{p}d\lambda(z)\right]^{\frac{1}{p}}
\left[\ind  |\varphi(z)|^{q}d\lambda(z)\right]^{\frac{1}{q}}<\infty.
\end{eqnarray*}
Then we can define an analytic function $f_1$ in $\D$ by
$$
f_1(z)=(2\alpha+1)\ind \frac{\varphi(w)(1-|w|^{2})^{\alpha-1}}{(1-z\overline{w})^{2+2\alpha}} dA(w).
$$
Let $f_2$ be the function with $f_2''=f_1$. Then
\begin{eqnarray} \label{2103}
&~&(1-|z|^{2})^{\alpha+1}f_2''(z) \nonumber \\
&= &(1-|z|^{2})^{\alpha+1} (2\alpha+1) \ind \frac{(1-|w|^{2})^{\alpha-1}\varphi(w)}{(1-z\overline{w})^{2+2\alpha}}dA(w).
\end{eqnarray}
Since $-q(\alpha+1)<-1<q\alpha$, it follows from  Theorem \ref{opera inter} that
\begin{eqnarray*}
\ind \left|(1-|z|^{2})^{\alpha+1} f_2''(z)\right|^q d\lambda(z) \lesssim \ind  |\varphi(z)|^{q}d\lambda(z).
\end{eqnarray*}
Joining this with (\ref{bub2101}), we get
\begin{equation}\label{2104}
|||f_2|||_{B^q_\alpha} \lesssim \|F \circ G^{-1}\|   \lesssim \|G^{-1}\| \ \|F\|.
\end{equation}
By (\ref{2102}) and  (\ref{2103}), for $g\in \widetilde{B}_{\alpha}^{p}$, we see
\begin{eqnarray*}
&~&\ind m_{g, \alpha}''(z)\overline{f_2''(z)}(1-|z|^2)^{2\alpha}dA(z)\\
&=& \ind (1-|w|^{2})^{\alpha-1}\overline{\varphi(w)} dA(w) (2\alpha+1) \ind \frac{m_{g, \alpha}''(z)(1-|z|^2)^{2\alpha}}{(1-\overline{z}w)^{2+2\alpha}} dA(z)\\
&=& \ind m_{g, \alpha}''(w) (1-|w|^{2})^{\alpha-1}\overline{\varphi(w)} dA(w)=F(g).
\end{eqnarray*}
Combining this with Lemma \ref{lem 41}, we get  $F(g)=\langle g, f_2 \rangle_\alpha$ for all $g\in \widetilde{B}_{\alpha}^{p}$.
We also  have proven that  $F(g)= \langle g,f\rangle_\alpha$ for all $g\in B_\alpha^p$. Bear in mind (\ref{b281}). Then  $f(z)-f(0)-f'(0)z=f_2(z)-f_2(0)-f_2'(0)z$.   Due to
(\ref{2104}),   $f\in B^q_\alpha$ and $|||f|||_{B^q_\alpha}   \lesssim   \|F\|$.  By  this and (\ref{2101}),  we obtain   (\ref{231}). The proof is complete.
\end{proof}

\vspace{0.1truecm}
\noindent {\bf Acknowledgements. }

The authors want to thank Professor Kehe Zhu  for interesting discussions on the subject.

\vspace{0.1truecm}
\noindent {\bf Data Availability. }

All data generated or analyzed during this study are included in this article and in its bibliography.

\vspace{0.1truecm}
\noindent {\bf Conflict of Interest. }

The authors declared that they have no conflict of interest.

\end{document}